\documentclass[10pt,a4paper]{article}
\usepackage[utf8x]{inputenc}
\usepackage{ucs}
\usepackage{amsmath}
\usepackage{amsfonts}
\usepackage{amssymb}
\usepackage{amsthm}
\usepackage{multirow}
\usepackage{float}
\usepackage{enumerate}
\author{Marek Ha{\l}enda\footnote{
Institute of Mathematics, Faculty of Mathematics, Physics and Informatics, University of Gda\'nsk, 
80-308 Gda\'nsk, Poland.\newline E-mail: mhalenda@mat.ug.edu.pl}}
\title{Complex Hantzsche-Wendt manifolds}
\date{}
\theoremstyle{plain}
\newtheorem{thm}{Theorem}
\newtheorem{lem}{Lemma}
\newtheorem{cor}{Corollary}
\newtheorem{prop}{Proposition}
\theoremstyle{definition}
\newtheorem{defn}{Definition}
\newtheorem{remark}{Remark}
\newtheorem{ex}{Example}
\DeclareMathOperator{\Q}{\mathbb{Q}}
\DeclareMathOperator{\R}{\mathbb{R}}
\DeclareMathOperator{\Z}{\mathbb{Z}}

\DeclareMathOperator{\C}{\mathbb{C}}
\begin{document}
\maketitle
\begin{abstract}
Complex Hantzsche-Wendt manifolds are flat K\"ahler manifolds with holonomy group $\Z_2^{n-1}\subset SU(n)$. They are important examples of Calabi-Yau manifolds of abelian type. In this paper we describe them as quotients of a product of elliptic curves by a finite group $\tilde{G}$. This will allow us to classify all possible integral holonomy representations and give an algorithm classifying their diffeomorphism types.

\vspace{1em}
\end{abstract}

\section{Introduction}
The class of complex Hantzsche-Wendt manifolds was defined in \cite{kahlerflat} as a complex analogue to the class of Hantzsche-Wendt manifolds. The latter is a class of flat orientable $n$-dimensional manifolds with holonomy group $\Z_2^{n-1}$, a generalisation of the original Hantzsche-Wendt manifold which is the unique $3$-dimensional manifold with holonomy group $\Z_2^2$. Any Hantzsche-Wendt manifold is odd-dimensional, thus it cannot be a complex manifold. However, we use the following definition: 
\begin{defn}
A flat K\"ahler manifold of complex dimension $n$ is called a complex Hantzsche-Wendt manifold (abbreviated as CHW-manifold) if the holonomy group of $M$ is isomorphic to $\Z_2^{n-1}$ and contained in the $SU(n)$ group.
\end{defn}

Let us briefly recall some basic properties of complex Hantzsche-Wendt manifolds:
\begin{itemize}
\item The holonomy group of a CHW manifold is the subgroup of diagonal integer matrices in $SU(n)$.
\item The complex dimension of CHW manifold is an odd number. Moreover, it is not hard to construct an example of CHW manifold in any odd dimension  (using results on the existence of Hantzsche-Wendt manifold, see \cite[Lemma 4.2]{kahlerflat}). 
\item The Hodge numbers of a CHW manifold $M$ are given by:
$$h^{p,q}(M)=\begin{cases}
\binom{n}{p}&\textrm{if }p=q\textrm{ or }p+q=n\\
0&\textrm{otherwise}
\end{cases}$$
(for the calculation see \cite[Theorem 4.3]{kahlerflat}).
\item Any CHW manifold is an algebraic variety (which follows from the equation $h^{2,0}=0$).
\end{itemize}

All three-dimensional CHW manifolds were classified (up to diffeomorphism) in \cite{kahlerflat}. However they were also found independently by different authors. In particular they were listed in \cite{oguiso}  as examples of Calabi-Yau threefolds of Type A (more detailed description, which agree with our results, can be found in \cite{kanazawa}). We point out that there are many inequivalent definitions of Calabi-Yau manifolds in literature and that the CHW manifolds do not satisfy all of them. However, it is interesting to observe, that the Hodge numbers of CHW manifolds satisfy equations $h^{n,0}=h^{0,0}=1$ and $h^{p,0}=0$ for $p=1,2,\ldots,n-1$, which is typical to the Calabi-Yau manifolds.

In this paper we describe the structure of CHW manifolds of any complex dimension $n$ as orbit spaces of a free action of a finite group on the product of $n$ elliptic curves. Details of that description will be given in Theorem \ref{structurethm}. In particular for $n=3$ it agrees with \cite[Proposition~6.2, pt 1.]{kanazawa} and \cite[Theorem~5.5]{lange}. Next we show in Corollary \ref{holonomystruct} that the integral holonomy representation decomposes as direct product of two subrepresentations, of which the first is diagonal and the second is isomorphic to the diagonal representation over $\Q$ but not necessarily over $\Z$. We classify representations of the second type in Proposition \ref{prop-sub_n} using equivalence classes of some subgroups of $\Z_2^n$ (using different terminology, isometry classes of binary linear codes of length $\geq 2$). After that, for any CHW manifold $M$ we define two matrices $\Phi(M)$ and $\Psi(M)$ that encode the objects from which $M$ is build up (see Definition \ref{defpsi} and Definition \ref{defphi}). We examine which pairs of matrices $\Phi$, $\Psi$ correspond to a CHW manifold (Theorem \ref{torsionfreethm}) and we define an equivalence relation (see Definition \ref{equivpairs}) such that pairs $(\Phi,\Psi)$ and $(\Phi',\Psi')$ are equivalent if and only if corresponding manifolds are diffeomorphic (Theorem \ref{equivthm}). This way we obtain an algorithm of classification of fundamental groups of CHW manifolds, which is analogue to the algorithm of classification of Hantzsche-Wendt manifolds given in \cite{miatello-rosetti}. At the end of this paper we present results obtained using our algorithm in complex dimension 5. There are exactly 8616 diffeomorphism classes of CHW fivefolds. We were not able to determine exact number of CHW manifolds in dimension seven due to hardware restrictions. However, by analysis of our method we find a lower bound (equal to 48321790784) on the number of CHW 7-folds with diagonal holonomy representation. It is interesting to compare this results with the results on number of Hantzsche-Wendt manifolds: there are 2 five-dimensional and 62 seven-dimensional HW manifolds (\cite{miatello-rosetti,mr7}).

Results of this paper were published in the author's PhD Thesis (\cite{phd}, in Polish language only).

Let us review some basic facts on flat manifolds, which will be needed further. A flat manifold $M$ is a compact Riemannian manifold with the Riemann curvature tensor equal to zero. It is a quotient space $\R^n/\Gamma$, where $\Gamma=\pi_1(M)$ is a torsion-free, cocompact and discrete subgroup of the group $E(n)=O(n)\ltimes \R^n$ of isometries of $\R^n$. Those groups are called Bieberbach groups (torsion-free crystallographic groups). They are an essential tool in studying flat manifolds since they determine them up to diffeomorphism. The algebraic structure of crystallographic groups is well understood: by first theorem of Bieberbach any crystallographic subgroup $\Gamma$ of $E(n)$ fits into short exact sequence:
$$0\longrightarrow \Lambda \longrightarrow \Gamma \longrightarrow H\longrightarrow 1,$$
where $\Lambda\subset\R^n$ is a maximal abelian subgroup of $\Gamma$ and is a lattice in $\R^n$ (isomorphic to $\Z^n$) and $H=\Gamma/\Lambda$ is a finite group called holonomy group of $\Gamma$. If $\Gamma$ is Bieberbach group, then $H$ is the holonomy of manifold $\R^n/\Gamma$. The short exact sequence above allows to define faithful integral representation of the group $H$ (which will be reffered to as the integral holonomy representation). The second Bieberbach theorem asserts that two crystallographic groups $\Gamma_1$, $\Gamma_2$ are isomorphic if and only if they are conjugated in the group $A(n)=GL_n(\R)\ltimes \R^n$. It follows, that if $\Gamma_1$ and $\Gamma_2$ are isomorphic, then the images of integral holonomy representations of $\Gamma_1$ and $\Gamma_2$ are conjugated in $GL_n(\Z)$ (though representations are not necessarily isomorphic).

Flat K\"ahler manifolds are characterised by a number of equivalent conditions (\cite[Theorem 3.1 and Proposition 3.2]{johnson}). In particular, $\Gamma$ is fundamental group of a flat K\"ahler manifold of complex dimension $n$ if and only if it is discrete, cocompact, torsion-free subgroup of $U(n)\ltimes \C^n$. 

\section{Structure theorem}
We are going to describe the structure of complex Hantzsche-Wendt manifolds. Assume that $M$ is such manifold. Recall that $M$ is a quotient of the form $M=\C^n/\Gamma$, where $\Gamma$ is a Bieberbach group. We have short exact sequence:
$$0\longrightarrow\Lambda\longrightarrow\Gamma\longrightarrow\Z_2^{n-1}\longrightarrow1$$
where $\Lambda\subset \Gamma$ is the maximal abelian subgroup of $\Gamma$ isomorphic to $\Z^{2n}$. Since $\Lambda$ is a lattice in $\C^n$, then $T=\C^n/\Lambda$ is a complex torus and $M$ is a quotient $M=T/\Z_2^{n-1}$ (where the group $\Z_2^{n-1}$ acts biholomorphically). We shall show that $T$ is isogenous to the product of elliptic curves. Consequently $M$ is a quotient $M=\hat{T}/\hat{G}$ where $\hat{T}=E_1\times E_2\times\ldots\times E_n$ is a product of some elliptic curves and $\hat{G}$ is an extension of $\Z_2^{n-1}$.

Let us begin with some notation. Let $\varrho:\Z_2^{n-1}\to \mathrm{SL}(V)$ be a faithful representation, where $\dim_{\C}V=n$. The image of $\varrho$ (up to conjugation) is the subgroup of diagonal integer matrices of $SU(n)$ (see \cite[proof of Lemma 4.2]{kahlerflat}). Thus we have decomposition $V=V_1\oplus V_2\oplus\ldots\oplus V_n$. Moreover we can find generators $g_1,g_2\ldots, g_{n}$ of $\Z_2^{n-1}$ such that $V_i=\ker(g_i-\mathrm{id})$ for all $i\in\{1,\ldots,n\}$.

\begin{lem}
Let $\Lambda\subset V$ be a lattice (of rank $2n$) and $\Lambda_i=V_i\cap \Lambda$. Then $E_i=V_i/\Lambda_i$ are elliptic curves for $i=1,2,\ldots,n$.\label{lem1}
\end{lem}
\begin{proof}
Let us consider maps $1+g_i:V\to V$ where $1$ is the identity map and $1\leq i\leq n$. It is easy to see that for all $i$ we have $\mathrm{Im}(1+g_i)=V_i$ and  $(1+g_i)(\Lambda)\subset\Lambda$. Take $v\in V_i$. Since $\Lambda$ spans $V$ we have $v=\displaystyle\sum_{l=1}^{2n} a_l\lambda_l$ for some elements $\lambda_1,\ldots,\lambda_{2n}\in\Lambda$ and $a_1,\ldots,a_{2n}\in\R$. On the other hand, $v\in V_i$ implies $(1+g_i)(v)=2v$. Thus we get $v=\displaystyle\sum_{l=1}^{2n}\frac{a_l}{2}(1+g_i)(\lambda_l)$. This means that elements $(1+g_i)(\lambda_l)$ for $l=1,2,\ldots,2n$ spans $V_i$, in other words $(1+g_i)(\Lambda)$ is a lattice in $V_i$. Since $(1+g_i)(\Lambda)\subset \Lambda_i$, we get our result.
\end{proof}
\begin{lem}
Let $T=V/\Lambda$ and $E_i=V_i/\Lambda_i$ as in previous lemma. Let us define a map $\mu: \bigoplus_{i=1}^n E_i\to T$ by: $$\mu(t_1,t_2\ldots,\ t_n)=t_1+t_2+\ldots+t_n$$
Then $\mu$ is an isogeny. Moreover elements $(t_1,t_2,\ldots,t_n)\in\ker\mu$ are 2-torsion points which cannot have exactly one non-zero entry $t_i$.\label{lem2}
\end{lem}
\begin{proof}
To be more precise, if $v_i\in V_i$ for $i=1,\ldots,n$, then $\mu(v_1+\Lambda_1,\ldots,v_n+\Lambda_n) = v_1+\ldots+v_n+\Lambda$. Clearly, since $V=V_1\oplus\ldots\oplus V_n$, then $\mu$ is an epimorphism.

Assume that $x=(0,\ldots,0,t,0,\ldots,0)\in\ker\mu$. If $v\in V_i$ such that $t=v+\Lambda_i$ for some $i$, then from definition of $\mu$ we have $v\in\Lambda$. But now $v\in V_i\cap\Lambda$ and $t=0$.

To show that any element $(t_1,\ldots,t_n)$ of the kernel of $\mu$ is a 2-torsion point, let us observe that for any $i=1\ldots,n$ we have $((1+g_i)t_i,\ldots,(1+g_i)t_n)=(0,\ldots,0,2t_i,0,\ldots,0)\in\ker\mu$. Thus $2t_i=0$.
\end{proof}

Now we are ready to prove our structure theorem for CHW manifolds.

\begin{thm}
Let $n$ be an odd number. A manifold $M$ is a $n$-dimensional CHW manifold if and only if there exist:
\begin{itemize}
\item elliptic curves $E_1,\ldots,E_n$,
\item non-zero 2-torsion points $t_i\in E_i$, $i=1,2,\ldots,n$,
\item a subgroup $\displaystyle W\subset\langle t_1,\ldots,t_n\rangle\subset \bigoplus_{i=1}^n E_i$ such that $t_i\not\in W$ for $i=1,2,\ldots,n$,
\item elements $\displaystyle(x_i^1,\ldots,x_i^i,\ldots,x_i^n)\in\bigoplus_{i=1}^n E_i$, where $x_i^i$ is a non-zero 2-torsion point and $x_i^i\neq t_i$,
\end{itemize}
which satisfy the following conditions:
\begin{enumerate}[i)]
\item for all $i\neq j$ elements $x_{ij}=(x_{ij}^1,\ldots,x_{ij}^n)\in W$ where each entry $x_{ij}^k$ is defined as:
$$x_{ij}^k=\begin{cases}0&k=i\textrm{ or }k=j\\2(x_i^k-x_j^k)&k\not\in\{i,j\}\end{cases}$$\label{strthm-cond1}
\item for all subsets $I=\{i_1,i_2,\ldots,i_m\}\subsetneq\{1,2,\ldots, n\}$ such that $m$ is odd and for any $(y_1,\ldots,y_n)\in W$ $\exists\ j\in I$ such that $x_I^j\neq y_j$, where
$$x_I^j=\begin{cases}\sum_{k=1}^m(-1)^{k-1}x_{i_k}^j&j\not\in I\\\sum_{k=1}^p(-1)^{k-1}x_{i_k}^j-\sum_{k=p+1}^m(-1)^{k-1}x_{i_k}^j&j=i_p\in I\end{cases},$$\label{strthm-cond2}
\item for any $j\in\{1,2,\ldots,n\}$ equality $x_n^j=x^j_{\{1,2,\ldots,n-1\}}$ holds,\label{strthm-cond3}
\end{enumerate}
for which $M$ is of the form:
$$M=E_1\oplus E_2\oplus\ldots\oplus E_n/\hat{G}$$
Here $\hat{G}=\langle \hat{f}_i,t_w\ :\ 1\leq i\leq n, w\in W\rangle$ where maps $\hat{f}_i$ are given by equations $\hat{f}_i(z_1,\ldots,z_n)=(-z_1+x_i^1,\ldots,z_i+x_i^i,\ldots,-z_n+x_i^n)$ and $t_w$ are translations of $\bigoplus_{i=1}^n E_i$ by all elements  $w\in W$.
\label{structurethm}
\end{thm}
\begin{proof}
First observe that $\hat{G}$ is finite. Since $W$ consists of 2-torsion points, then any translation $t_w$ by an element of $W$ is an involution commuting with maps $\hat{f}_i$ for $i=1,\ldots,n$. Moreover, since $x_i^i$ is a 2-torsion point, $\hat{f}_i$ also is an involution. Condition iii) ensures that $\hat{f}_n=\hat{f}_1\hat{f}_2\ldots\hat{f}_{n-1}$ and from condition i) any commutator $[\hat{f}_i,\hat{f}_j]$ is a translation $t_w$ for some $w\in W$. Thus $\langle t_w\ :\ w\in W\rangle\simeq\Z_2^d$ and $\hat{G}$ is a central extension of this group by a group $\Z_2^{n-1}$. 

Now we shall show that the orbit space $\bigoplus_{i=1}^n E_i/\hat{G}$ is a manifold. It is enough to show that elements of $\hat{G}$ have not fixed points. Let $\hat{g}=t_w\hat{f}_{i_1}\hat{f}_{i_2}\ldots\hat{f}_{i_k}$ be an arbitrary element of $\hat{G}$. Using relations in $\hat{G}$ we can assume that $i_1<i_2<\ldots<i_k$ and that $k$ is odd. Denote the set $\{i_1,\ldots,i_k\}$ by $I$. Let also $y_i$ denote the coordinates of $w$, i.e. $w=(y_1,y_2,\ldots,y_n)$. For any $z=(z_1,\ldots,z_n)$ we have: 
$$\hat{g}(z)=((-1)^{\alpha_1}z_1+x_I^1+y_1,\ldots,(-1)^{\alpha_n}z_n+x_I^n+y_n),$$
where $\alpha_i=0$ if $i\in I$ and $\alpha_i=1$ if $i\not\in I$. If $\hat{g}$ had a fixed point $z$, then for all $i\in I$ we would have $x_I^i+y_i=0$. However, we assumed the contrary in the condition ii).

From the definition of $M$ it is clear now, that $M$ is a complex flat manifold. It is also clear that the holonomy group of $M$ is isomorphic to the group generated by linear parts of the maps $\hat{f_i}$ (that is to the group $\Z_2^{n-1}$) and that this group consists of matrices of determinant one. Thus $M$ is a CHW manifold. 

Now assume that $M$ is a CHW manifold. Let $M=V/\Gamma$, where $\Gamma=\pi_1(M)$ and let $\Lambda\subset\Gamma$ be a maximal lattice. Denote $T=V/\Lambda$. From the definition of CHW manifold we have decomposition $V=V_1\oplus V_2\oplus\ldots\oplus V_n$ which allows us to define elliptic curves $E_i$ as quotients $V_i/\Lambda_i$ where $\Lambda_i=\Lambda\cap V_i$ (Lemma \ref{lem1}). We also know from Lemma \ref{lem2} that the kernel of the isogeny $\mu:\bigoplus_{i=1}^nE_i\to T$ given by $\mu(z_1,z_2,\ldots,z_n)=\sum_{i=1}^nz_i$ consists of 2-torsion points. This kernel will coincide with the group $W$. However, at this moment we only see that it is a subgroup of group of 2-torsion points which is isomorphic to $\Z_2^{2n}$.

We can choose $n-1$ elements $(g_1,x_1),\ldots,(g_{n-1},x_{n-1})$ of the group $\Gamma$ where the linear parts $g_i$ are defined by diagonal matrices $$\mathrm{diag}(-1,\ldots,-1,1,-1,\ldots,-1)$$ (where the only one $1$ entry is on $i$-th place) in a basis corresponding to the decomposition of $V$. Then we define $(x_i^1,\ldots,x_i^n)$ as an arbitrarily chosen element of $\mu^{-1}(x_i+\Lambda)$. In case $i=n$ we use condition iii) to define elements $x_n^j$, then $\mu(x_n^1,\ldots,x_n^n)=x_n$ where $(g_n,x_n)=(g_1,x_1)\ldots(g_{n-1},x_{n-1})$.

For any $i$, map $g_i$ yields an automorphism $\tilde{g}_i$ of the torus $T$. Similarly we have biholomorphic maps $f_i:T\to T$ given by $f_i(v+\Lambda)=\tilde{g}_i(v+\Lambda)+\tilde{x_i}$ where $\tilde{x}_i=x_i+\Lambda$. Then group $\langle f_i\ :\ 1\leq i\leq n\rangle$ is isomorphic to $\Z_2^{n-1}$. Thus for all $i$ we have $f_i^2=\mathrm{id}$ and hence $\tilde{g}_i(\tilde{x}_i)+\tilde{x}_i=0$. Using  $\tilde{x}_i=x_i^1+\cdots+x_i^n$ we get $2x_i^i=0$. Observe that from Lemma \ref{lem2} this equality holds also in $E_i$. In similar way from relation $f_if_j=f_jf_i$ we get $(\mathrm{id}-\tilde{g}_i)(\tilde{x}_j)=(\mathrm{id}-\tilde{g}_j)(\tilde{x}_i)$. This yields:
$$2\sum_{k\neq i,j} (x_i^k-x_j^k)=0$$
Hence $x_{ij}\in\ker\mu$, where $x_{ij}$ is an element defined as in condition i).

Next, since $\Gamma$ is torsion-free, then the biholomorphic maps $f_{i}$ have no fixed points (if any of them had a fixed point, then it would be easy to find an element of $\Gamma$ which does not act freely). A map $f:T\to T$ given by equation $f(z)=g(z)+x$ (where $g$ is an automorphism of $T$) has no fixed points if and only if $x\not\in \mathrm{im}(\mathrm{id}-g)$. Observe that any element of the group generated by automorphisms $\tilde{g}_i$ is of the form $\tilde{g}_I=\tilde{g}_{i_1}\tilde{g}_{i_2}\ldots \tilde{g}_{i_k}$, where set $I=\{i_1,i_2,\ldots,i_k\}$ is odd (this follows from relation $\tilde{g}_n=\tilde{g_1}\ldots\tilde{g}_n$). If $z=\mu(z_1,\ldots,z_n)$ then $(\mathrm{id}-\tilde{g}_I)(z)=\sum_{i\not\in I}2z_i$. Hence we get $\tilde{x}_I\not\in \mu(\bigoplus_{i\not\in I}E_i)$. This is equivalent to the condition ii) from the definition of elements $x_I^j$ (because $\mu(x_I^1,\ldots,x_I^n)=\tilde{x}_I$).

Applying above condition to the special case $I=\{i\}$ we get, that there is no element $(y_1,\ldots,y_n)\in\ker\mu$ such that $y_i=x_i^i$. Thus $x_i^i\neq 0$. It follows also that for any $i$ and for any two elements $(y_1,\ldots,y_n)$ and $(y_1',\ldots,y_n')$ of $\ker\mu$ if both $i$-th entries $y_i$ and $y_i'$ are non-zero, then they are equal. This allows us to define elements $t_i\in E_i$ and shows that $W=\ker\mu$ is a subgroup of a group isomorphic to $\Z_2^{n-1}$.

To finish the proof, observe that directly from the definition of $\hat{G}$ we get $M=V/\Gamma=T/\Lambda=T_1\times\cdots\times T_n/\hat{G}$.
\end{proof}

\section{Integral holonomy representation}
It is known that Hantzsche-Wendt manifolds have diagonal integral holonomy (\cite[Theorem 3.1]{generalizedHW}). This cannot be true for complex Hantzsche-Wendt manifolds. Indeed, let $n$ be an odd number and $\varrho_{diag}$ be a faithful representation of $\Z_2^{n-1}$ in the $SL_n(\Z)$ such that its image consists of diagonal matrices. We take any representation $\varrho$ of $\Z_2^{n-1}$ such that it is isomorphic to $\varrho_{diag}$ over $\Q$ but not over $\Z$. Then we are able to construct Hantzsche-Wendt manifold with integral holonomy representation $\varrho\oplus\varrho_{diag}$. We shall show that the structure theorem implies that any integral holonomy representation of CHW manifold is of that form. We will denote by $L_{diag}$ a module corresponding to the representation $\varrho_{diag}$.
\begin{cor}
Let $M=V/\Gamma$ be a CHW manifold of dimension $n$ and $\Lambda\subset\Gamma$ the maximal lattice. Then $\Lambda=L_1\oplus L_2$ as a $\Z_2^{n-1}$-module, where $L_1\simeq L_{diag}$ and $L_1\otimes_{\Z}\Q=L_2\otimes_{\Z}\Q$.\label{holonomystruct}
\end{cor}
\begin{proof}
Let $\Lambda'\subset\Lambda$ be a sublattice defined as $\Lambda'=\bigoplus_{i=1}^n\Lambda_i$, where $\Lambda_i$ are defined as in Lemma \ref{lem1}. For any $i$, lattice $\Lambda_i$ has two generators which we denote by $\lambda_i^1$ and $\lambda_i^2$. This gives us following decomposition:
$$\Lambda'=\langle\lambda_i^1\ |\ 1\leq i\leq n\rangle\oplus\langle\lambda_i^2\ |\ 1\leq i\leq n\rangle,$$
where both summands are diagonal modules. Now observe that the group of 2-torsion points in $T_i=V_i/\Lambda_i$ is generated by $\frac{1}{2}\lambda_i^1 +\Lambda_i$ and $\frac{1}{2}\lambda_i^2 +\Lambda_i$. Without loss of generality we assume that $x_i^i=\frac{1}{2}\lambda_i^1 +\Lambda_i$ and $t_i= \frac{1}{2}\lambda_i^2 +\Lambda_i$. Since $\Lambda$ is generated by elements of $\Lambda'$ and by elements of $W=\ker \mu$, then we have decomposition:
$$\Lambda=\langle\lambda_i^1\ |\ 1\leq i\leq n\rangle\oplus\langle\lambda_i^2, w\ |\ 1\leq i\leq n, w\in W\rangle.$$
The first component is the diagonal module. The second component is diagonal over $\Q$ because the set of elements $\{\lambda_i^2\ :\ 1\leq i\leq n\}$ is its basis.
\end{proof}
First step in order to classify fundamental groups of CHW manifolds is to find all integral representations of odd degree $n$ (in fact, their images in $GL(n,\Z)$). It is possible to do it in terms of the subgroup $W$.
\begin{defn}
Let $G=\langle t_1,t_2,\ldots,t_n\rangle\simeq\Z_2^n$. We introduce equivalence relation on the set of subgroups of $G$: if $W_1<G$ and $W_2<G$ then $W_1\sim W_2 \iff \exists\ \sigma\in S_n$ such that the map $h_\sigma:G\to G$ defined by $h(t_i)=t_{\sigma i}$ restricts to an isomorphism from $W_1$ to $W_2$. 

Then we define:
$$\mathrm{Sub}(n)=\left\{W<G\ |\ \forall\ i\in\{1,2,\ldots,n\}\ t_i\not\in W\right\}/\sim$$\label{sub_n}
\end{defn}

\begin{prop}
Classes of subgroups in the set $\mathrm{Sub}(n)$ are in one-to-one correspondence with the conjugacy classes in $GL(n,\Z)$ of images of integral representations which are equivalent to the diagonal representation $\varrho_{diag}$ over~$\Q$.\label{prop-sub_n}\end{prop}
\begin{proof}

Let $L$ be a $\Z_2^{n-1}$-module corresponding to the integral representation $\varrho:\Z_2^{n-1}\to GL(n,\Z)$, such that $L\otimes_{\Z}\Q\simeq L_{diag}\otimes_{\Z}\Q$. Fix a set $g_1,\ldots,g_{n-1}$ of generators of $\Z_2^{n-1}$ in a such way, that for any $g_i$ its eigenspace (over $\Q$) of $1$ is one-dimensional. Denote also $g_n=g_1g_2\ldots g_n$. Define:
$$L_i=\{l\in L\ :\ g_i(l)=l\}$$
for $i=1,2,\ldots,n$. Now $(1+g_i)(L)\subset L_i$ and thus $\Lambda=L_1\oplus L_2\oplus\ldots\oplus L_n\subset L$ is the maximal submodule isomorphic to the diagonal module. Define $W=L/\Lambda$. It can be seen as a subgroup of the quotient $\frac{1}{2}\Lambda/\Lambda$. To show that, we use an argument similar to that used in Lemma \ref{lem2}: for every $l\in L$ and for every $i=1,2,\ldots,n$ we have $(1+g_i)l\in L_i\subset \Lambda$, and $\displaystyle\sum_{i=1}^n(1+g_i)l=2l$. Moreover for $i=1,2,\ldots n$ let $t_i$ be a coset of $\frac{1}{2}l_i$, where $l_i$ is a generator of $L_i$. Then  any of the $t_i$ cannot be an element of $W$, because $\Lambda$ is maximal.

Assume that $L'$ is a module isomorphic to $L$ via isomorphism $\varphi:L\to L'$. Then since $\varphi$ restricts to isomorphism of the maximal diagonal submodules $\Lambda\subset L$ and $\Lambda'\subset L'$, we have $L/\Lambda=L'/\Lambda'$. On the other hand, if $W$ is a representative of a class $[W]\in\mathrm{Sub}(n)$ and $L_{diag}$ is a diagonal module, then there exists unique module $L\subset \frac{1}{2}L_{diag}$ such that $L/L_{diag}=W$. 

Assume that $L'$ is not neccessarily isomorphic to $L$, but the images of representations $\varrho$ and $\varrho'$ corresponding to $L$ and $L'$ are conjugate. Then there exists a permutation $\sigma\in S_n$ inducing an automorphism $f_{\sigma}$ of $\Z_2^{n-1}$ defined on generators by an equation $f_{\sigma}(g_i)=g_{\sigma(i)}$ such that $\varrho'\circ f_{\sigma}$ is isomorphic to $\varrho$. Thus $L'/\Lambda'$ and $L/\Lambda$ represent the same class in $\mathrm{Sub}(n)$. Conversely, if $W$ and $W'$ are representatives of the same class $[W]=[W']\in\mathrm{Sub}(n)$, then there exists $\sigma$ such that $\sigma W=W'$. Let $\varrho$ and $\varrho'$ be a representations associated to $W$ and $W'$. Changing the order of generators according to $\sigma$ we get an isomorphism of these representations, thus their images are conjugate.

\end{proof}
\begin{remark}
We can think of subgroups of $\Z_2^{n-1}$ as subspaces of the vector space $F_2^k$. Such subspaces are also known as binary linear codes. Recall that the Hamming distance of two elements of such a code is the number of positions where these elements differ. The distance of a linear code is the minimal Hamming distance between two different elements of this code. A linear isometry of codes is a linear map between codes that preserves the Hamming distance. Using this language, the set $\mathrm{Sub}(n)$ is the set of linear isometry classes of $(n,*,\geq 2)$ binary codes (where $*$ can be any number between $0$ and $n-1$). 

There is a notion of dual code (subspace of vectors orthogonal to a given code  with respect to the standard bilinear form). It is easy to see that a code dual to code of distance $\geq 2$ is a code that has no column of zeros in its generator matrix. Such codes were studied in \cite{slepian}. In particular, using notation from that paper, $|\mathrm{Sub}(n)|=\displaystyle\sum_{k=1}^{n-1}S_{nk}$. Further informations on enumeration and generation of such codes can be found in \cite{codes}.
\end{remark}

\begin{ex}
In dimension $3$ we have four classes in $\operatorname{Sub}(n)$:
\begin{enumerate}
\item trivial class, corresponding to the diagonal module,
\item two classes of order $2$ subgroups ($1$-dimensional codes):
\begin{itemize}
\item class of subgroup generated by $(1,1,0)$ (code of distance $2$),
\item class of subgroup generated by $(1,1,1)$ (code of distance $3$),
\end{itemize}
\item one class of order $4$ subgroup ($2$-dimensional code), which is generated by $(1,1,0)$ and $(0,1,1)$.
\end{enumerate}
For each of this classes there is only one manifold with corresponding holonomy representation. 
\end{ex}
\section{Fundamental groups}
Now we are in position to classify CHW manifolds up to diffeomorphism, in other words, to classify their fundamental groups. We will construct an algorithm which can be used to list all such Bieberbach groups. It can be also used to check if two CHW manifolds are diffeomorphic. The algorithm will be similar to the algorithm for HW manifolds obtained in \cite{miatello-rosetti}.

Assume that $M$ is a CHW manifold. According to Theorem \ref{structurethm}, we can associate to $M$ a collection of elements $x_i^j$, $t_i$ for $i,j=1,2,\ldots,n$ and a subgroup $W$ of group generated by elements $t_i$. Recall that for all $i\neq j$ an element $x_{ij}$ (defined in Theorem \ref{structurethm}) is an element of $W$. Thus for every $k=1,2,\ldots n$ we have $x_{ij}^k=2(x_i^k-x_j^k)\in\{0,t_k\}$. This allows us to introduce the following definition.
\begin{defn}
Let $M$ be a CHW manifold and let $\{x_i^j, t_i\ |\ 1\leq i,j\leq n\}$ be a set of elements associated to $M$ as in Theorem \ref{structurethm} (we use also notation $x_{ij}^k$ from that Theorem). Then we define $\Psi(M)$ as a binary $n\times n$ matrix with coefficients $\psi_{ij}$ given by:
$$\psi_{ij}=\begin{cases}
0&\textrm{ if }x_{ik}^j=0,\\
1&\textrm{ if }x_{ik}^j=t_j,\\
\end{cases}$$
where $$k=\begin{cases}1&\textrm{ if }j>1\\2&\textrm{ if }j=1.\end{cases}$$
\label{defpsi}
\end{defn}
\begin{prop}
Let $M$ be a CHW manifold of complex dimension $n$ together with a fixed set of elements $x_i^j$ and $t_j$, for $1\leq i,j\leq n$ and a subgroup $W\subset\langle t_1,\ldots,t_n\rangle$, defined as in the proof of Theorem \ref{structurethm}. Then the matrix $\Psi(M)$ has following properties:
\begin{enumerate}
\item $\psi_{ii}=0$ for $i=1,2,\ldots,n$,
\item $\psi_{21}=\psi_{12}=\psi_{13}=\ldots=\psi_{1n}=0$,
\item for all $i\neq j$ there exists an element $w=(w_1,\ldots,w_n)\in W$ such that $w_i=w_j=0$ and for any $k\not\in\{i,j\}$ we have $w_k=t_k \iff \psi_{ik}\neq\psi_{jk}$,  
\item sum of the coefficients in every column is $0 \pmod 2$.
\end{enumerate}
\label{psi-matrix-props}
\end{prop}
\begin{proof}
The first statement follows from Theorem \ref{structurethm}, as $x_i^i$ is a 2-torsion point. The second statement follows directly from the definition of the matrix $\Psi(M)$ and the definition of $x_{ij}^k$. The remaining statements are easily implied by conditions \ref{strthm-cond1}) (because we can take $w=x_{ij}$) and \ref{strthm-cond3}) of Theorem \ref{structurethm}. 
\end{proof}
\begin{cor}
If CHW manifold $M$ is a threefold, then $\Psi(M)$ is a zero matrix. If $M$ is any CHW manifold such that the associated subgroup $W$ has at most two elements (in other words, binary linear code corresponding to the holonomy representation is at most one-dimensional) then $\Psi(M)$ is a zero matrix.
\end{cor}
\begin{proof}
From the two first statements of the Proposition \ref{psi-matrix-props}, if $n=3$ then $\Psi(M)$ has only three coefficients: $\psi_{23}, \psi_{31}$ and $\psi_{32}$ which could be non-zero. However as there is only one such an element in every column of $\Psi(M)$, from the last statement of Proposition \ref{psi-matrix-props} we conclude that $\Psi(M)=0$.

Assume now that $W$ has at most 2 elements. If it has only $1$ element (that is $W$ is trivial) then obviously $\Psi(M)=0$. So assume that $|W|=2$ and let $w=(w_1,\ldots,w_n)$ be the non-zero element of $W$. Let $i$ be an index such that $w_i\neq 0$. Then if we take any $j\neq i$ and use third statement from the Proposition \ref{psi-matrix-props}, then we get that for any $k\not\in\{i,j\}$ we have $\psi_{ik}=\psi_{jk}$ (as the only one element of $W$ with $0$ on the $i$-th coordinate is the trivial element). This shows that for any $k\neq i$, in $k$-th column of matrix $\Psi(M)$ all elements but the $i$-th and the $k$-th are the same. The $k$-th element in the $k$-th column is always $0$ (this is the first statement in Proposition \ref{psi-matrix-props}). Since $n$ is an odd number and all elements in every column of $\Psi$ sum up to $0$, then all elements in every column (but the $i$-th) which are not on the diagonal are equal. From the second statement of Proposition \ref{psi-matrix-props} we see that these columns consist only of zeros. Since we could find another number $l\neq i$ such that $w_l\neq 0$, then the $i$-th column has the same property and $\Psi$ is a zero matrix.
\end{proof}

Now we shall define second matrix associated to a CHW manifold. We will denote by $\Z_2[\tau]$ an additive group on elements $\{0,1,\tau,1+\tau\}$ isomorphic to the Klein 4-group. If $M$ is a CHW manifold and $E_1,\ldots,E_n$ are elliptic curves defined as in Lemma \ref{lem1}, then for each $i=1,2,\ldots,n$ we define isomorphisms $\kappa_i$ from the group $E_i(2)$ of $2$-torsion points of $E_i$ to the group $\Z_2[\tau]$ by the values on generators:
$$\kappa_i(a)=\begin{cases}1&\textrm{ if }a=x_i^i\\\tau&\textrm{ if }a=t_i\end{cases}$$

\begin{defn}Let $M$ be a CHW manifold of complex dimension $n$ with associated collection of objects $\{x_i^j,t_i\ |\ 1\leq i,j\leq n\}$ as in Theorem \ref{structurethm}. Let $\Psi(M)$ be the matrix defined in the Definition \ref{defpsi}. Then $\Phi(M)$ is a $n\times n$ matrix with coefficients $\phi_{ij}\in\Z_2[\tau]$ defined as follows:
$$\phi_{ij}=\begin{cases}
1&\textrm{ if }i=j,\\
\kappa_j(x_i^j-x_k^j)&\textrm{ if }i\neq j,
\end{cases}$$
where $k\in\{1,2,\ldots,n\}$ is the smallest number such that $k\neq j$ and $\psi_{kj}=\psi_{ij}$.\label{defphi}
\end{defn}

Observe, that since $\psi_{kj}=\psi_{ij}$, then $2(x_i^j-x_k^j)=2(x_i^j-x_l^j)-2(x_k^j-x_l^j)=0$, where $l=1$ for $j>1$ and $l=2$ for $j=1$. Thus matrix $\Phi(M)$ is well defined.

\begin{prop}
Let $M$ be a CHW manifold of complex dimension $n$ together with a fixed set of elements $x_i^j$ and $t_j$, for $1\leq i,j\leq n$. Then the matrix $\Phi(M)$ has following properties:
\begin{enumerate}
\item $\phi_{21}=\phi_{12}=\phi_{13}=\ldots=\phi_{1n}=0$,
\item if $k_j$ is the smallest index such that $\psi_{k_jj}=1$ then $\phi_{k_jj}=0$,
\item for all $j\in\{1,2,\ldots,n\}$ sum of the elements of the $j$-th column of $\Phi(M)$ is given by the equation:
$$\sum_{i=1}^n \phi_{ij}=\frac{1}{2}\left(\sum_{i=1}^{j-1}(-1)^{i-1}\psi_{ij}-\sum_{i=j+1}^{n}(-1)^{i-1}\psi_{ij}\right)\tau.$$
\end{enumerate}\label{phi-matrix-props}
\end{prop}
\begin{proof}
The first two properties follows directly from the definition of~$\Phi(M)$. The last property is implied by the condition \ref{strthm-cond3}) of the Theorem~\ref{structurethm}. Indeed, we have:
\begin{equation*}\begin{split}0&=\sum_{i=1}^{j-1}(-1)^{i-1}x_i^j+x_i^i-\sum_{i=j+1}^{n}(-1)^{i-1}x_i^j\\
&=x_i^i+\sum_{i=1}^{j-1}(-1)^{i-1}(x_i^j-x_l^j)-\sum_{i=j+1}^{n}(-1)^{i-1}(x_i^j-x_l^j),\end{split}
\end{equation*}
where $l=2$ if $j=1$ and $l=1$ otherwise. Let $k$ be the minimal index such that $\psi_{kj}=1$. If $i\neq j$, then $(x_i^j-x_l^j)-\psi_{ij}(x_k^j-x_l^j)=\kappa_j^{-1}(\phi_{ij})$. We have also that $x_j^j=\kappa_j^{-1}(\phi_{jj})$. Thus we get:
$$\left(\sum_{i=1}^{j-1}(-1)^{j-1}\psi_{ij}-\sum_{i=j+1}^{n}(-1)^{j-1}\psi_{ij}\right)(x_l^j-x_k^j)=\sum_{i=1}^n \kappa_j^{-1}\phi_{ij}.$$
But from definition of $\Psi(M)$, equality $\psi_{kj}=1$ means that $2(x_k^j-x_l^j)=t_j=\kappa_j^{-1}(\tau)$. Because $\kappa_j$ is an isomorphism, we get our result. 
\end{proof}

\begin{defn}
Let $n$ be an odd number, and $W$ be a representative of a class $[W]\in\textrm{Sub}(n)$. Let
$$\mathcal{M}\subset\textrm{Mat}_{n\times n}(\Z_2[\tau])\times \textrm{Mat}_{n\times n}(\Z_2)$$
be a subset consisting of pairs of matrices that satisfy the conditions from Proposition \ref{psi-matrix-props} and Proposition \ref{phi-matrix-props}. We say that two pairs $(\Phi,\Psi),(\Phi',\Psi')\in\mathcal{M}$ are equivalent, if $(\Phi,\Psi)$ can be transformed to $(\Phi',\Psi')$ by a finite number of  operations:
\begin{enumerate}
\item change pair $(\Phi,\Psi)$ to $(\sigma\Phi,\sigma\Psi)$ for any $\sigma\in S(W)= \{\sigma \in S_n: \sigma(W) = W\}$,
\item add a vector $\tau w$ to one of the rows of $\Phi$ (but not to the last row), where $w=(w_1,\ldots,w_n)\in W\subset \Z_2^n$ and $\tau w=(\tau w_1,\ldots,\tau w_n)$,
\item apply to the $k$-th column of matrix $\Phi$ the bijection $\gamma$ of $\Z_2[\tau]$ which is a transposition of elements $\tau$ and $1+\tau$, where $k\in\{1,2,\ldots,n\}$ is any number such that for any $w=(w_1,\ldots,w_n)\in W$ we have $w_k=0$,
\end{enumerate}
where each of the operations is followed by a normalization (if neccessary):
\begin{enumerate}
\item change all coefficients $\psi_{ij}$ of a transformed matrix $\Psi$ such that $i\neq j$ to $\psi_{ij}-\psi_{lj}$, where $l=1$ for $j>1$ and $l=2$ for $j=1$,
\item change all coefficients $\phi_{ij}$ of a transformed matrix $\Phi$ such that $i\neq j$ to $\phi_{ij}-\phi_{kj}$, where $k$ is the smallest number such that $k\neq j$ and $\psi_{ij}=\psi_{kj}$,
\item change elements $\phi_{nj}$ for $j<n$ and element $\phi_{n-1\ n}$ so that the sum of every column of $\Phi$ is equal to $\frac{1}{2}\left(\sum_{i=1}^{j-1}(-1)^{i-1}\psi_{ij}-\sum_{i=j+1}^{n}(-1)^{i-1}\psi_{ij}\right)\tau$,
\item if any element $\phi_{ii}$ on the diagonal of transformed matrix $\Phi$ is not equal to $1$, then apply to all elements of the $i$-th column of $\Phi$ the bijection $\delta$ of the $\Z_2[\tau]$ which is a transposition of $1$ and $1+\tau$.
\end{enumerate}\label{equivpairs}
\end{defn}
\begin{remark}
Observe, that the only operation which affects matrix $\Psi$ is the first one. After this operation it may be necessary to use the first three normalizations (of which only the first one may change the matrix $\Psi$). Thus the first operation together with the first normalization defines an action of the group $S(W)$ on the set of possible matrices $\Psi$. For a fixed matrix $\Psi$ we define a group $S(\Psi)$ as a stabilizer of $\Psi$ under this action.\label{S_psi}
\end{remark}
\begin{thm}
Let $M, M'$ be a CHW manifolds of complex dimension $n$ with isomorphic integral holonomy representations corresponding to $W\subset \Z_2^n$. Assume that $\Phi(M)$, $\Psi(M)$, $\Phi'(M)$ and $\Psi'(M)$ are matrices associated to $M$ and $M'$ respectively. Then $M$ is diffeomorphic to $M'$ if and only if the pairs $(\Phi(M),\Psi(M))$ and $(\Phi(M'),\Psi(M'))$ are equivalent (in the sense of Definition~\ref{equivpairs}).\label{equivthm} 
\end{thm}
\begin{proof}
At the beginning of the proof we shall show that a pair of matrices $(\Phi(M),\Psi(M))$ determines a diffeomorphism type of $M$. Assume that $M$ and $M'$ are CHW manifolds of complex dimension $n$ (with the same holonomy representation corresponding to $W$) to which we associated the same matrices, $\Phi(M)=\Phi(M')$ and $\Psi(M)=\Psi(M')$. These matrices were defined using some objects, which we will denote $x_i^j, t_j\in E_j$ (for manifold $M$) and $(x_i^j)', (t_j)'\in (E_j)'$ (for manifold $M'$), where $1\leq i,j\leq n$. Let $M=V/\Gamma$ and $M'=V/\Gamma'$, where $\Gamma$ and $\Gamma'$ are the fundamental groups of $M$ and $M'$. Since the holonomy representation is fixed, then without loss of generality we can assume that $\Lambda=\Lambda'$, where $\Lambda\subset\Gamma$ and $\Lambda'\subset\Gamma'$ are the maximal lattices. Moreover it follows that for any $i=1,2,\ldots,n$ we have $E_i=(E_i)'$. Up to diffeomorphism we may also assume, that $x_i^i=(x_i^i)'$ and $t_i=(t_i)'$, for all $i=1,2,\ldots,n$. Let $l_j=1$ if $j>1$, $l_1=2$. We will show that without any loss of generality we may assume that $x_{l_j}^j=(x_{l_j}^j)''$. Indeed, if it is not the case, take $a\in V$ such that the coset of $a$ in $T=V/\Lambda$ is equal to
$$\frac{1}{2}\mu\left((x_{l_1}^1)'-x_{l_1}^1,\ldots,(x_{l_n}^n)'-x_{l_n}^n\right).$$
Let $\Gamma''=\Gamma^{(Id,a)}$. Then $V/\Gamma''$ is diffeomorphic to $M$. Observe, that if $(g_i,x_i)\in\Gamma$ is an element defined as in the proof of Theorem \ref{structurethm}, then $(Id,a)(g_i,x_i)(Id,-a) = (g_i,x_i+(id-g_i)a)\in\Gamma''$. Thus we can choose new elements $(x_i^j)''$ in the following way:
$$(x_i^j)''=\begin{cases}x_i^i&i=j\\x_i^j+(x_{l_j}^j)'-x_{l_j}^j&i\neq j\end{cases}$$
From the definition we see that $\Psi(M)=\Psi(M'')$ and $\Phi(M)=\Phi(M'')$. Moreover $(x_{l_j}^j)''=(x_{l_j}^j)'$.

Now it follows from definition of $\Phi(M)$ that if $\psi_{ij}=0$, then $x_i^j=(x_i^j)'$. To end this part of the proof, we need to show that we can assume $x_{k_j}^j=(x_{k_j}^j)'$, where $k_j$ is the smallest number such that $\psi_{k_jj}=1$. We use the fact that $x_{k_jl_j}^j=2(x_{k_j}^j-x_{l_j}^j)=t_j$. Thus the difference $x_{k_j}^j-x_{l_j}^j$ can be equal $\pm \frac{1}{2}t_j$ or $x_j^j\pm\frac{1}{2}t_j$. Consider the matrices:
$$A_1=\begin{bmatrix}3&4\\2&3\end{bmatrix},\ 
A_2=\begin{bmatrix}1&2\\0&1\end{bmatrix},\ 
A_3=\begin{bmatrix}3&2\\4&3\end{bmatrix}.$$
Lattice $\Lambda_j$ has a basis $a,b$ such that $\frac{1}{2}a+\Lambda=x_j^j$ and $\frac{1}{2}b+\Lambda=t_j$. Using this basis we can see that each of the above defined matrices induces a diffeomorphism of the elliptic curve $E_j$ to itself. These diffeomorphisms fix all $2$-torsion points. Moreover for any two points $x,y$ from the set $\{\frac{1}{2}t_j$,  $-\frac{1}{2}t_j$, $x_j^j+\frac{1}{2}t_j, x_j^j+\frac{1}{2}t_j\}$ we can find $i$ such that diffeomorphism induced by $A_i$ maps $x$ to $y$. Thus up to diffeomorphism we can assume that $x_i^j=(x_i^j)'$ for any $i,j\in\{1,2,\ldots,n\}$, but then $M=M'$.

In the second part of the proof we shall check that any of operations introduced in Definition \ref{equivpairs} does not change the diffeomorphism type of $M$. 
\begin{enumerate}
\item Observe that the first operation corresponds to the change of ordering of elliptic curves $E_i$ according to the permutation $\sigma$. The first two normalizations may be needed since the matrices $\Phi$ and $\Psi$ are defined using the differences of the elements $x_i^j$. 
\item The second operation changes the values $x_i^j$. However it does not change $\mu(x_1^j,\ldots,x_n^j)$ since $W=\ker\mu$. This means that under this operation the fundamental group $\Gamma$ is unchanged. Here it may be necessary to use all but the first normalization. We will need the second normalization for the same reason as above. Third normalization must be applied if we add to the $i$-th row of $\Phi(M)$ a vector $\tau w$ such that $w_i\neq 0$. In this case the element $x_i^i$ would be changed to $x_i^i+t_i$, but then we would rename the 2-torsion points of $E_i$ (see the proof of Theorem \ref{structurethm}). This renaming is precisely the third normalization. Observe also that adding a vector to a single row of $\Phi(M)$ corresponds to a different choice of elements $x_i^j$, and this implies also that elements $x_n^j$ would be chosen differently (which follows from the condition \ref{strthm-cond3}) of the Theorem \ref{structurethm}). 
\item Now we will consider the remaining operation. Let $k$ be a number such that for any $w\in W$ we have $w_k=0$. Then there are two $2$-torsion points which could be named $t_k$ equally well and the third operation corresponds to the different choice of this element. Clearly here we do not need any normalization.
\end{enumerate}

The coefficients of $\Psi(M)$ and $\Phi(M)$ are determined by the points $x_i^j$ (and implicitly by $t_j$). These points are not uniquely determined though. In the next step we shall show that different choice of points $x_i^j$ gives rise to equivalent pair of matrices. There are two cases. First assume, that $k$ is a number such that there exists $w\in W$ with $w_k\neq 0$. This determines the choice of an element $t_k$. However, since the $n$-tuple $(x_k^1,x_k^2,\ldots,x_k^n)$ is determined up to elements of $\ker\mu$, then there are two such $n$-tuples  which differ on the $k$-th place by $t_k$. This means that there are two possible choices of $x_k^k$. Transition between them can be done precisely by the second operation (adding a vector to a row of $\Phi$) together with the fourth normalization. In the second case, for all $w\in W$ we have $w_k=0$. Then the point $x_k^k$ is uniquely determined, but $t_k$ could be equally well substituted by $t_k+x^k_k$. This corresponds to the third operation from Definition \ref{equivpairs}.

Now we need to check that if $M=V/\Gamma$ and $M'=V/\Gamma'$ are diffeomorphic CHW manifolds and if $(\Psi(M),\Phi(M)),(\Psi(M'),\Phi(M'))$ are pairs of matrices associated to them, then these pairs are equivalent. From the Bieberbach Theorem, there is an element $(A,a)\in A(2n)$ such that $\Gamma'=\Gamma^{(A,a)}$. Since  it can be decomposed to the linear and translational part, then it is enough to check two cases separately: conjugation by translation and conjugation by linear map. 
\begin{itemize}
\item If $a=0$, then it follows from direct calculation that $\Lambda'=A\Lambda$. Moreover since the holonomy representations are isomorphic, then $A$ induce an isomorphism $\phi_A$ of elliptic curves $E_i$ and $E_i'$. Thus we can assume that elements $(x_i^j)'$ for $M'$ are chosen as $\phi_A(x_i^j)$ and that $t_j'=\phi_A(t_j)$. But this means that $(\Phi(M),\Psi(M))$ is a pair that could be associated to $M'$ as well. Thus $(\Phi(M'),  \Psi(M'))$ is equivalent to $(\Phi(M),\Psi(M))$.
\item Assume that $A=Id$. Since $(g_i,x_i)^{(Id,a)}=(g_i,x_i+(Id-g_i)a)$, then directly from the definition of matrices $\Phi$ and $\Psi$ we see, that also in this case a pair $(\Phi(M),\Psi(M))$ could be associated to $M'$.
\end{itemize}
\end{proof}

The theorem above can be used to check whether two CHW manifolds are diffeomorphic in terms of matrices $\Phi(M)$ and $\Psi(M)$. However, in order to classify all diffeomorphism types of $n$-dimensional CHW manifolds for a given odd number $n$, we need to know which pairs of matrices from the set $\mathcal{M}$ (defined in the Definiton \ref{equivpairs}) correspond to a CHW manifold. For this reason we introduce following condition.
\begin{defn}
Let $n$ be an odd number and $(\Phi,\Psi)\in \mathcal{M}$ (where the set $\mathcal{M}$ is defined as in the Definition \ref{equivpairs}). Let $\phi_{ij}$ and $\psi_{ij}$ denote the coefficients of matrices $\Phi$ and $\Psi$ respectively. We say that $(\Phi,\Psi)$ satisfies the torsion-free condition, if for any $I=\{i_1,\ldots,i_m\}\subset\{1,2,\ldots,n\}$ of odd size $m$ there exists $j=i_p\in I$ such that:
$$\psi_{Ij}\textrm{ is odd or }\sum_{k=1}^m\phi_{i_kj}\neq \frac{1}{2}\psi_{Ij}\tau,$$
where $\displaystyle\psi_{Ij}\in\Z$ is a number defined by equation: $$\psi_{Ij}=\sum_{k=1}^{p-1}(-1)^{k-1}\psi_{i_kj}-\sum_{k=p+1}^m(-1)^{k-1} \psi_{i_kj}.$$\label{torsionfreecond}
\end{defn}
\begin{thm}
Let $n$ be an odd number and $(\Phi,\Psi)\in\mathcal{M}$. Then the following conditions are equivalent:
\begin{enumerate}
\item there exists a complex Hantzsche-Wendt manifold $M$ such that $\Phi=\Phi(M)$ and $\Psi=\Psi(M)$,
\item all pairs $(\Phi',\Psi')\in\mathcal{M}$ equivalent to $(\Phi,\Psi)$ satisfy the torsion-free condition.
\end{enumerate}\label{torsionfreethm}
\end{thm}
\begin{proof}
First assume, that $M$ is a CHW manifold such that $\Phi=\Phi(M)$ and $\Psi=\Psi(M)$. Then, by Theorem \ref{equivthm}, any equivalent matrices  $\Phi',\Psi'$ are of the form  $\Phi'=\Phi(M')$ and $\Psi'=\Psi(M')$ for some manifold $M'$ diffeomorphic to $M$ (and for some choice of points $(x_i^j)'$ and $t_j'$). Thus it is enough to show that the pair $(\Phi,\Psi)$ satisfies the torsion-free condition. Let $x_i^j$ and $t_j$ be points chosen as in Theorem \ref{structurethm}. From condition \ref{strthm-cond2}) of this Theorem we know, that for any proper subset $I\subset\{1,2,\ldots,n\}$ of odd size and any $w=(w_1,\ldots,w_n)\in W$ we have $x_I^j\neq w_j$. In particular we can take $w=0$ and thus $x_I^j\neq 0$. But:
$$x_I^j=\kappa_j^{-1}\left(\sum_{k=1}^m\phi_{i_kj}\right)+\psi_{Ij}(x_{kj}-x_{lj})$$
(by direct calculation similar to the one used in the proof of Proposition \ref{phi-matrix-props}). If $\psi_{Ij}$ is even, then we get
$$0\neq\kappa_j^{-1}\left(\sum_{k=1}^m\phi_{i_kj}\right)+\frac{1}{2}\psi_{Ij}\kappa_j^{-1}\left(2(x_{kj}-x_{lj})\right),$$
and thus the torsion-free condition is satisfied.

Now we assume that all equivalent pairs to $(\Phi,\Psi)$ satisfy the torsion-free condition. To construct a CHW manifold $M$ we use matrices $\Phi$ and $\Psi$. Precisely, we take any elliptic curves $E_1,E_2,\ldots,E_n$, and for any $i  \in\{1,2,\ldots,n\}$ we choose two different 2-torsion points $x_i^i,t_i\in E_i$. Then by Theorem \ref{structurethm} we need also a group $W$ which is the subgroup to which the matrices $(\Phi,\Psi)$ are associated, and points $x_i^j$ which we obtain in following way:
$$x_i^j=\kappa_j(\phi_{ij})+\psi_{ij}y_j,$$
where $y_j\in E_j$ is a point such that $2y_j=t_j$ (which we may choose arbitrarily but only once). Conditions \ref{strthm-cond1} and \ref{strthm-cond3} of Theorem \ref{structurethm} are then easily satisfied from the properties of matrices $\Psi$ and $\Phi$. The only thing we need to prove is that the condition \ref{strthm-cond2} is also satisfied. Assume the contrary, that there exists $w=(w_1,w_2,\ldots,w_n)$ and a subset $I\subset\{1,2,\ldots n\}$ of odd size such that for every $j\in I$ we have $x_I^j=w_j$. Observe that in particular $2x_I^j=0$ and thus $\psi_{Ij}$ cannot be odd. We can get a new pair $(\Phi',\Psi')$ equivalent to $(\Phi,\Psi)$ by adding $w$ to $k$-th and $l$-th row of $\Phi$, where $k\in I$ and $l\not\in I$. First observe, that before doing normalization, we obtain a matrix which does not satisfy the torsion-free condition -- it follows from the assumption that condition \ref{strthm-cond2} of Theorem \ref{structurethm} is not satisfied. Now we will check, that after normalizations the torsion-free condition still is not satisfied. Let $\tilde{\Phi}$ denote matrix $\Phi$ with $w$ added to the $k$-th and $l$-th row, without normalizations and let $\tilde{\phi}_{ij}$ denote its $(i,j)$-coefficient. Fix $j\in I$. Let $\alpha,\beta$ and $\gamma$ be a numbers of coefficients $\tilde{\phi}_{ij}$ equal to $1$, $\tau$ and $1+\tau$ respectively, where $i\in I$. Since the torsion-free condition is not satisfied for the subset $I$, we have that $\alpha+\gamma$ and $\beta+\gamma+\frac{1}{2}\psi_{Ij}$ are even numbers. Let us examine all types of normalizations.
\begin{enumerate}
\item The first normalization operates on $\Psi$ which is unchanged.
\item The second normalization adds $\tau$ to some even numbers of elements, because $I$ is odd, we do not change elements on the diagonal and $\psi_{Ij}$ is even. Thus the parity of $\alpha+\gamma$ and $\beta+\gamma+\frac{1}{2}\psi_{Ij}$ is unchanged.
\item Since $w$ has been added to two rows, then the sum of the coefficients in every column remains the same. It follows that the third type of normalization is not necessary.
\item The last type of normalization swaps numbers $\alpha$ and $\gamma$. It is enough to observe that $\alpha+\gamma+\frac{1}{2}\psi_{Ij}$ is also an even number.
\end{enumerate}
We have arrived to a contradiction. Thus elements $x_i^j$ satisfy all conditions of Theorem \ref{structurethm} and we conclude that $M$ is a CHW manifold such that $\Phi(M)=\Phi$ and $\Psi(M)=\Psi$.
\end{proof}
\section{Quantitative results}
Theorems \ref{equivthm} and \ref{torsionfreethm} can be used to classify fundamental groups of complex Hantzsche-Wendt manifolds of given dimension $n$. We sum it up in the following steps:
\begin{enumerate}
\item Find all classes $[W]\in \mathrm{Sub}(n)$ (see Definition \ref{sub_n}).
\item For every class $[W]$:
\begin{enumerate}
\item choose a representative $W$ of a class $[W]$,
\item calculate the stabilizer $S(W)=\{\sigma\in S_n\ |\ \sigma W=W\}$,
\item find all possible matrices $\Psi$ (see Proposition \ref{psi-matrix-props}),
\item find the orbits of the action of $S(W)$ on the set of matrices $\Psi$ and choose a representative in every orbit,
\item for the representative $\Psi$ of every orbit:
\begin{enumerate}
\item calculate $S(\Psi)\subset S(W)$ (see Remark \ref{S_psi}),
\item find all possible matrices $\Phi$ (see Proposition \ref{phi-matrix-props}),
\item find the equivalence classes of pairs $(\Phi,\Psi)$ according to the relation introduced in Definition \ref{equivpairs} (here $\Psi$ is fixed so we use only permutations from $S(\Psi)$),
\item find those classes in which any pair $(\Phi,\Psi)$ satisfy the torsion-free condition (see Definition \ref{torsionfreecond}).
\end{enumerate}
\end{enumerate}
\end{enumerate}

As we have seen, there are four classes in $\mathrm{Sub}(3)$. Moreover the matrix $\Psi$ is a zero-matrix and there is unique matrix $\Phi$ satisfying the conditions of Proposition \ref{phi-matrix-props}. This confirms that there are four CHW threefolds, with distinct integral holonomy representations (\cite{kahlerflat}).

The set $\mathrm{Sub}(5)$ consists of 16 classes. Using a computer program based on our method, we can list all fundamental groups of CHW manifolds of complex dimension $5$. The total number of CHW fivefolds is equal to 8616. More detailed results are listed in Table \ref{5dimtable}. 

In the case of dimension $7$ and the diagonal module, we can use our algorithm to find a lower bound on number of fundamental groups. However, this is probably very rough estimate. 
\begin{prop}
The number of fundamental groups of $7$-dimensional CHW manifolds with diagonal holonomy representation exceeds $48321790784$.
\end{prop}
\begin{proof}
Integral holonomy representation of CHW manifold is diagonal if and only if the associated subspace $W$ is trivial. This implies that $\Psi$ is zero matrix. Equivalence of pairs $(\Phi,\Psi)$ reduces to an action of the group $S_7\times\Z_2^7$ on the set of matrices $\Phi$ (where $S_7$ acts via the first kind and $\Z_2^7$ via the third kind of operations from Definition \ref{equivpairs}). Since $\Phi$ has always $1$ on the diagonal, $0$ in the top non-diagonal entry and the sum in the bottom non-diagonal entry, then the set of possible matrices $\Phi$ is parametrized by values of the remaining $28$ coefficients, which have values in $\Z_2[\tau]$. Some of the matrices from that set do not satisfy the torsion-free condition. However, looking at the proof of Theorem \ref{equivthm}, we see that in our case if a matrix $\Phi$ satisfy the torsion-free condition, then it is also satisfied by all equivalent matrices. 

We will estimate how many possible matrices do not satisfy the torsion-free condition. We have 35 subsets of $\{1,2,\ldots,7\}$ of cardinality $3$ and 21 subsets of cardinality $5$. The torsion-free condition is not satisfied for the subset $I$ when the coefficients of $\Phi$ satisfy a set of $|I|$ linearly independent linear equations. Thus there are at most $35\cdot 4^{28-3}+28\cdot 4^{28-5}$ matrices for which the torsion-free condition is not satisfied. As a consequence the number of matrices $\Phi$ satisfying the torsion-free condition is greater than $4^{28}-35\cdot 4^{25}+28\cdot 4^{23}=2^{46}\cdot 443$. As the order of the group acting on those matrices is equal to $7!\cdot 2^7$, the number of its orbits exceeds $2^{46}\cdot 443/(7!\cdot 2^7)=2^{35}\cdot 443/315$.
\end{proof}

\begin{table}
\caption{Numbers of fundamental groups of CHW fivefolds.}
\vspace{1cm}
\begin{tabular}[!h]{|c|c|r|}
\hline
Generators of $W$ & $\Psi$ & Number\\
\hline
$(0,0,0,0,0)$& 0& 667\\\hline
$(1,1,0,0,0)$& 0& 1551\\\hline
$(1,1,1,0,0)$& 0& 1716\\\hline
$(1,1,1,1,0)$& 0& 1290\\\hline
$(1,1,1,1,1)$& 0& 420\\\hline
\multirow{2}{*}{$\left(\begin{array}{ccccc}
1&1&0&0&0\\1&0&1&0&0
\end{array}\right)$}&0&283\\\cline{2-3}
&$\Psi_1$&82\\\hline
$\left(\begin{array}{ccccc}
1&1&0&0&0\\0&0&1&1&0
\end{array}\right)$&0&516\\\hline
$\left(\begin{array}{ccccc}
1&1&0&0&0\\1&0&1&1&0
\end{array}\right)$&0&691\\\hline
$\left(\begin{array}{ccccc}
1&1&0&0&0\\0&0&1&1&1
\end{array}\right)$&0&381\\\hline
$\left(\begin{array}{ccccc}
1&1&0&0&0\\1&0&1&1&1
\end{array}\right)$&0&340\\\hline
$\left(\begin{array}{ccccc}
1&1&1&0&0\\1&0&0&1&1
\end{array}\right)$&0&362\\\hline
\multirow{3}{*}{$\left(\begin{array}{ccccc}
1&1&0&0&0\\1&0&1&0&0\\1&0&0&1&0
\end{array}\right)$}&0&34\\\cline{2-3}&$\Psi_1$&42\\\cline{2-3}&$\Psi_2$&25\\\hline
\multirow{3}{*}{$\left(\begin{array}{ccccc}
1&1&0&0&0\\1&0&1&0&0\\0&0&0&1&1
\end{array}\right)$}&0&56\\[0.5em]\cline{2-3}&$\Psi_1$&28\\[-0.5em]&&\\\hline
\multirow{3}{*}{$\left(\begin{array}{ccccc}
1&1&0&0&0\\1&0&1&0&0\\0&0&1&1&1
\end{array}\right)$}&0&41\\[0.5em]\cline{2-3}&$\Psi_1$&21\\[-0.5em]&&\\\hline
\multirow{3}{*}{$\left(\begin{array}{ccccc}
1&1&0&0&0\\0&0&1&1&0\\1&0&1&0&1
\end{array}\right)$}&0&38\\[0.5em]\cline{2-3}&$\Psi_3$&9\\[-0.5em]&&\\\hline
\multirow{4}{*}{$\left(\begin{array}{ccccc}
1&1&0&0&0\\1&0&1&0&0\\1&0&0&1&0 \\1&0&0&0&1
\end{array}\right)$}&0&2\\\cline{2-3}&$\Psi_1$&9\\\cline{2-3}&$\Psi_2$&9\\\cline{2-3}&$\Psi_4$&4\\\hline

\end{tabular}
\qquad\qquad
\begin{tabular}{c}
$\Psi_1= \begin{bmatrix}0&0&0&0&0\\0&0&0&0&0\\0&0&0&0&0\\1&1&1&0&0\\1&1&1&0&0\end{bmatrix}$
\\$ $ \\
$\Psi_2=\begin{bmatrix}0&0&0&0&0\\0&0&0&0&0\\1&1&0&1&0\\1&1&1&0&0\\0&0&1&1&0\end{bmatrix}$
\\$ $ \\
$\Psi_3 = \begin{bmatrix}0&0&0&0&0\\0&0&0&0&0\\1&1&0&1&1\\1&1&1&0&1\\0&0&1&1&0\end{bmatrix}$
\\$ $ \\
$\Psi_4 = \begin{bmatrix}0&0&0&0&0\\0&0&1&1&0\\0&1&0&0&1\\1&0&1&0&1\\1&1&0&1&0\end{bmatrix}$
\end{tabular}
\label{5dimtable}
\end{table}

\bibliographystyle{plain}

\end{document}